\documentclass[preprint,compress,12pt]{elsarticle}

\usepackage{amssymb}
\graphicspath{{figures/}}
\usepackage{caption}
\usepackage{multirow}
\usepackage{threeparttable}
\usepackage{bm}
\usepackage{makecell}
\usepackage{booktabs}
\usepackage{subcaption}
\usepackage{comment}
\usepackage{url}
\usepackage{tabularx}
\usepackage{hyperref}
\usepackage{amsmath,amsfonts,amsthm}
\newtheorem{thm}{Theorem}[section]
\newtheorem{cor}[thm]{Corollary}

\theoremstyle{remark}

\journal{Elsevier}

\begin{document}

\begin{frontmatter}



\title{SPFNO: Spectral operator learning for PDEs with Dirichlet and Neumann boundary conditions}

\author{Ziyuan Liu\fnref{label1}}{\ead{liuziyuan17@nudt.edu.cn}}
\author{Yuhang Wu\fnref{label1}}{\ead{wuyuhang18@nudt.edu.cn}}
\author{Daniel Zhengyu Huang\fnref{label2}}{\ead{huangdz@bicmr.pku.edu.cn}}
\author{Hong Zhang\fnref{label1}}{\ead{zhanghnudt@163.com}}
\author{Xu Qian\fnref{label1}\corref{cor1}}{\ead{qianxu@nudt.edu.cn}}
\author{Songhe Song\fnref{label1}\fnref{label3}}{\ead{shsong@nudt.edu.cn}}

\cortext[cor1]{Corresponding author.}
\affiliation[label1]{organization={Department of Mathematics, National University of Defense Technology},
            city={Changsha},
            postcode={410073}, 
            country={China}}
\affiliation[label2]{organization={Beijing International Center for Mathematical Research, Peking University},
            city={Beijing},
            postcode={100871}, 
            country={China}}
\affiliation[label3]{organization={State Key Laboratory of High Performance Computing, National University of Defense Technology},
            city={Changsha},
            postcode={410073}, 
            country={China}}

\begin{abstract}
  Neural operators have been validated as promising deep surrogate models for solving partial differential equations (PDEs). Despite the critical role of boundary conditions in PDEs, however, only a
  limited number of neural operators robustly enforce these conditions. In this paper we introduce semi-periodic Fourier neural operator (SPFNO), a novel spectral operator learning method for solving PDEs with
  non-periodic BCs. This method extends our previous work (arXiv:2206.12698), which showed significant improvements by employing enhanced neural operators that precisely satisfy the boundary
  conditions. However, the previous work is associated with Gaussian grids, restricting comprehensive comparisons across most public datasets. In response to this, we present numerical results of
  SPFNO for various PDEs
  such as the viscous Burgers' equation, Darcy flow, incompressible pipe flow, and coupled reaction-diffusion equations. These results demonstrate the computational efficiency, resolution invariant property,
  and BC-satisfaction behavior of proposed model. An accuracy improvement of approximately 1.7X–-4.7X over the non-BC-satisfying baselines is also achieved. Furthermore, our studies on SOL underscore
  the significance of satisfying BCs as a criterion for deep surrogate models of PDEs.
\end{abstract}


\begin{highlights}
\item Precisely satisfies the boundary conditions.
\item State-of-the-art (SOTA) performance on 5 publicly available PDE datasets.
\item $O(N log N)$ time complexity and discretization invariance.
\item Flexibility to various tasks and general geometries.
\end{highlights}

\begin{keyword}
neural operator \sep deep learning-based PDE solver \sep AI4science \sep scientific machine learning \sep spectral method



\end{keyword}

\end{frontmatter}



\section{Introduction}
Partial Differential Equations (PDEs) play a pivotal role in various scientific and engineering fields, modeling phenomena such as heat conduction, fluid flow, electromagnetic waves, and quantum
mechanics. Given that a substantial majority of PDEs lack analytical solutions, numerical methods, such as spectral methods and finite difference methods, are the primary means of numerically solving
PDEs. Recently, researchers have discovered that deep-learning methods can serve as more convenient and efficient alternatives to these traditional methods. At present, two primary deep learning methodologies are employed for
solving PDEs, including directly approximating of the solution using neural networks~\cite{han2018solving,sirignano2018dgm,he2019mgnet,zhu2018bayesian,khoo2021solving,jin2020sympnets,lyu2022mim}, e.g., deep Ritz methods \cite{yu2018deep,weinan2019barron,duan2021convergence}
, physics informed neural networks (PINNs, \cite{raissi2019physics,wang2023solving,mao2020physics,gao2023failure,shin2020convergence,sun2020physics,jagtap2020adaptive,guo2022monte,jiao2023gas}), and extreme learning machine \cite{dong2022computing}; or approximating the nonlinear
operator between the input and output functions \cite{lu2019deeponet,gin2021deepgreen,fan2020solving,cai2021deepm,jin2022mionet,nelsen2021random,li2021fourier,li2020multipole,raonic2023convolutional,bartolucci2023representation}, which is known as operator learning  and is the focus of this paper.

The boundary condition (BC) of PDE plays a crucial role in defining the behavior of the system. Due to the significant physical information contained in BCs, there has been an emerging interest in
accurately enforcing these conditions when directly learning the solution using neural networks \cite{lu2021physics,horie2022physics,sukumar2022exact,SamuelJMLR}. Due to the complexity and architectural differences involved in learning a nonlinear operator, the direct generalization of these methods to neural operators is non-trivial.

\begin{figure}[h!]
  \centerline{\includegraphics[width=1\textwidth]{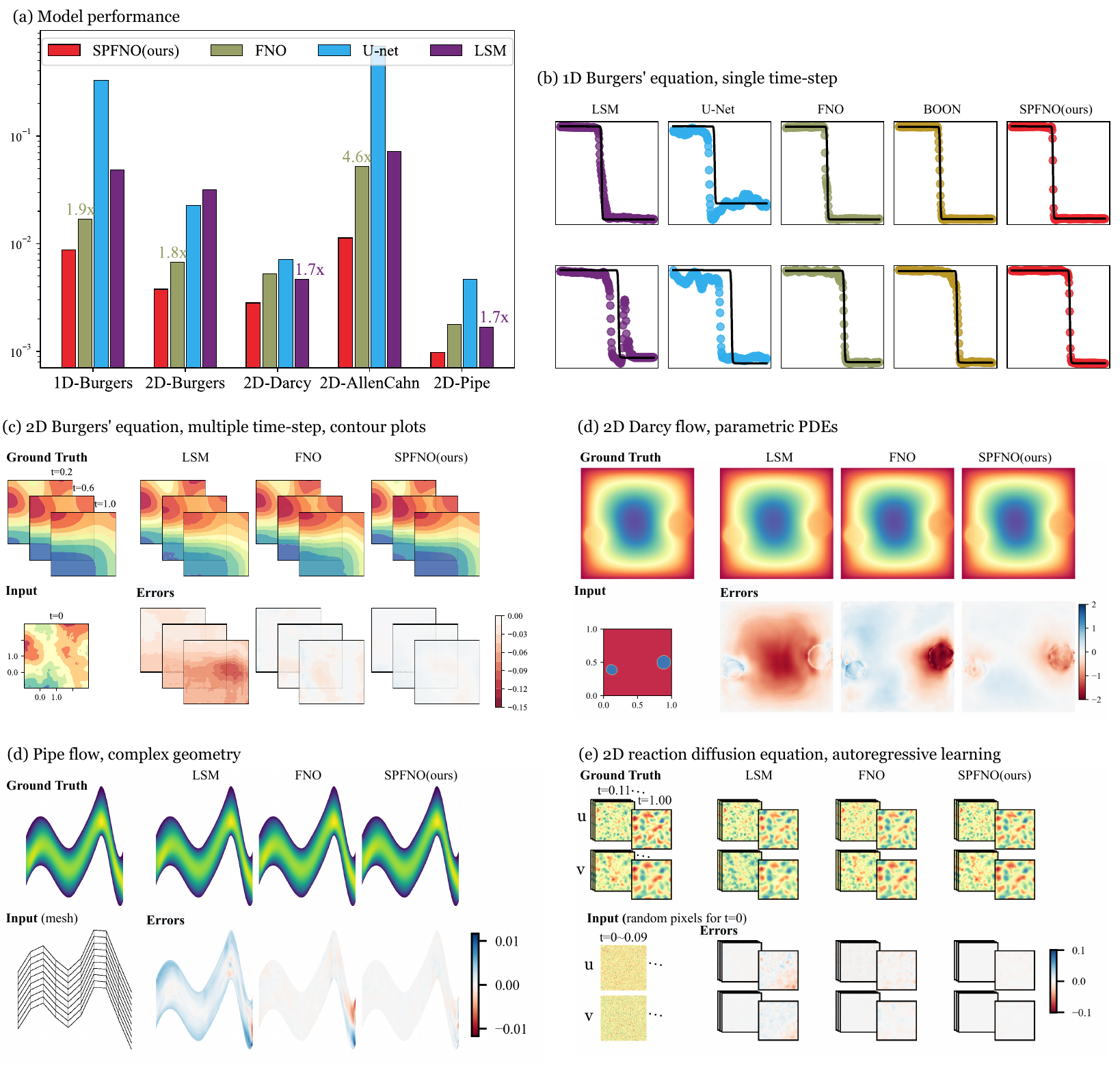}}
  \caption{(a) Comparisons of the accuracy between SPFNO and non-BC-satisfying baseline models. The lower, the better. (b)--(e) Examples of numerical experiments pertaining to various PDE learning tasks.}\label{fig:all}
\end{figure}

Fourier neural operator (FNO, \cite{li2021fourier,kovachki2021neural}) is a popular neural operator with applications in various fields \cite{pathak2022fourcastnet,zhang2022fourier,Grady2022ModelParallelFN,wen2022u,zhou2023ai}, followed
by many derivative neural operators that adopted a similar spectral analysis backbone and replaced the Fourier transform. While the treatment of BCs is also the most crucial issue for spectral
methods, it is hard for the majority of these neural
operators to accurately enforce the most commonly used BCs, namely the Dirichlet, Neumann, and Robin BCs (also known as the first-, second-, and third-type
BCs, respectively). By generalizing the spectral backbone of FNO using the technique of spectral methods with a suitable basis, we have introduced a general framework named spectral
operator learning (SOL) in \cite{liu2022render}, in which the enhanced neural operators satisfy the BCs exactly. And as the first SOL instance, the orthogonal polynomial neural operator (OPNO) showed
state-of-the-art performance on solving PDEs with non-periodic BCs. Moreover, when solving the heat transfer equation with Robin BCs, it acheived an unprecedent relative $L^2$ norm error of $1e-6$ in
the implementation of all neural operators. In addition, based on similar ideas, Bonev et al. \cite{bonev2023spherical} developed a spherical fourier neural operators (SFNO) that strictly satisfies behavioral boundary conditions
on the sphere \cite{boyd2001chebyshev}; while the Boundary enforcing Operator Network (BOON, \cite{saad2023guiding}) enforces the BCs to neural operators using a refinement procedure.

Meanwhile, there have been multiple studies on the application of the well-known
transformer models in operator learning \cite{cao2021choose,liu2022ht}, among which the recently developed latent spectral model (LSM, \cite{wu2023LSM}) ranked 1st in solving multiple PDE datasets. Thus,
the transformers are significant challengers to neural operators in solving PDEs.

Unfortunately, the fast algorithms for Chebyshev transformation in OPNO depend on the Gaussian grids, while typically, most public datasets only provide values on uniform grids, limiting the comprehensive comparison
between SOL architecture
with the non-BC-satisfying models. Introducing a novel SOL model with fast algorithms on a uniform grid and comparing it with state-of-the-art baselines across a wider range of datasets will undoubtedly
strengthen the persuasiveness of our viewpoint. To address this technical issue,  we introduce SPFNO, of which the \textbf{fast transformation algorithm} with a time complexity of $O(NlogN)$
is designed on a $N$-point uniform grid, to solve PDEs with Dirichlet and Neumann
BCs. The name SPFNO is derived from the \textbf{S}emi-\textbf{P}eriodic \textbf{F}ourier \textbf{N}eural \textbf{O}perator and also represents the \textbf{SP}ecified \textbf{F}ourier \textbf{N}eural \textbf{O}perator with non-periodic BCs. Its specified trigonometric bases allow the errors on the BCs to reach machine precision. In addition, SPFNO also possesses the following appealing properties that are expected from a neural
operator.
\begin{itemize}
\item \textbf{Invariance to discretization}. Without the need for retraining, an SPFNO model trained on a coarse grid can directly predict the solution on any fine grid. Detailed discussion on this property is conducted in Sec. \ref{sec:to-fine}.
\item \textbf{Efficient and accurate spatial differentiation}. Differentiating the output function of SPFNO requires operations of only $O(NlogN)$ using the spectral method.
\end{itemize}
Comparisons of the performance and several examples are given in Fig. \ref{fig:all}. All datasets used are already public, and the code, pre-trained model and reproducible training process are made publicly available at \href{https://github.com/liu-ziyuan-math/SPFNO}{https://github.com/liu-ziyuan-math/SPFNO}.

\section{Methodology}
\subsection{Spectral Operator Learning}\label{subsec:sol}
The SOL is a kind of specifically designed neural architecture \cite{liu2022render}, which involves learnable spectral operators that are induced by suitable basis function and implemented through correspoding
transformation. Simultaneously, it emphases the concept of strictly satisfying BCs in deep surrogate models for PDEs. This architecture allows for the utilization of various spectral methods techniques in deep learning methods and leads to improved accuracy and physical validity of the neural operator models.

We now briefly introduce this architecture through the task of prediction of time-dependent PDEs. However, its application is not limited to this scenario, and other cases can be
found in the numerical experiments and Fig. \ref{fig:all}. Consider the following PDE
\begin{equation}
u_t(x, t) + \mathcal N(u(x, t)) = 0, x \in \Omega, t \in [0, T] \nonumber
\end{equation}
with inital--boundary conditions
\begin{eqnarray}
           u(x, 0) = u_0(x) , x \in \Omega, \nonumber \\
  \mathcal B(u(x, t)) = 0, x \in \partial \Omega, \label{eq:3}
\end{eqnarray}
where $\mathcal N$ is a continuous operator and $\mathcal B$ is an operator corresponding to specific BCs. The task is to learn the solution operator $\mathcal S_{\tau}$ which evolves the initial condition $u_0$ to
the solution at $\tau$, namely, $\mathcal S_{\tau}(u_0(x)) = u(x, t=\tau)$. We let $\left\{ u_k(x) \right\}_{k \in \mathbb N}$ and $\mathcal T$ be a set of basis functions satisfying the BC (\ref{eq:3}) and the linear transform
induced by such a basis, respectively. Then the SOL model for $\mathcal S(\tau)$ consists of a stack of neural spectral layers
$\mathcal L^{(l)}$ that are induced by $\left\{ u_k \right\}$ and  $\mathcal T$, and
are in the form of
\begin{equation}
  u^{(l+1)} = \sigma(W_l u^{(l)} + \mathcal L^{(l)} u^{(l)}) := \sigma(W_l u^{(l)} + \mathcal T^{-1} A_l \mathcal T u^{(l)}), \label{eq:neural-operator}
\end{equation}
where $\sigma$ is the nonlinear activation function, $W_l$ is an auxiliary pointwise shallow neural network, and $A_l$ is a learnable spectrum-wise matrix.

The form of architecture \eqref{eq:neural-operator} is first demonstrated by Li et al. for FNO \cite{li2021fourier}, and then adopted by multiple neural operators, such as the multiwavelet-based neural operator (MWT-NO, \cite{gupta2021multiwavelet}),
wavelet neural operator (WNO, \cite{tripura2022wavelet}), spectral neural operator (SNO, \cite{fanaskov2022spectral}), orthogonal polynomial neural operator
(OPNO, \cite{liu2022render}), multi-channel IAE-Net \cite{ong2022iae} and so on. Some Examples are listed in Tab. \ref{tab:nos}. While matching the BCs by selecting a suitable underlying basis is a fundamental criterion in
spectral methods, this criterion was first introduced in
\cite{liu2022render} to operator learning, and we specifically refer to the those approach that inherently satisfies the BCs of the problem as spectral operator learning. So when solving PDEs with periodic BCs, FNO also falls into the category of SOL.

\begin{table}[h!]
\captionsetup{skip=1pt}
\caption{The fundamental spectral bases, the specified BCs and the types of utilized grids of neural operators that are based on spectral analysis. Only a subset of models are listed. When the FNO
  model is applied to periodic BCs, it also fulfills  the definition of the SOL architecture.}\label{tab:nos}
\begin{center}
\scriptsize
\begin{tabular}{|c|cccc|ccc|}
\hline
\multirow{2}{*}{Model} & \multicolumn{1}{c|}{\multirow{2}{*}{MWT-NO}} & \multicolumn{1}{c|}{\multirow{2}{*}{WNO}} & \multicolumn{1}{c|}{\multirow{2}{*}{IAE-Net}} & \multirow{2}{*}{FNO}                                            & \multicolumn{3}{c|}{\textbf{SOL models}}                                                                                                                                                                                                                      \\ \cline{6-8} 
                       & \multicolumn{1}{c|}{}                        & \multicolumn{1}{c|}{}                     & \multicolumn{1}{c|}{}                         &                                                                 & \multicolumn{1}{c|}{SFNO}                                                            & \multicolumn{1}{c|}{SPFNO}                                                                       & OPNO                                                                \\ \hline
Basis                  & \multicolumn{1}{c|}{multiwavelets}           & \multicolumn{1}{c|}{wavelets}             & \multicolumn{1}{c|}{integral}                 & Fourier                                                         & \multicolumn{1}{c|}{\begin{tabular}[c]{@{}c@{}}spherical\\ harmonics\end{tabular}}   & \multicolumn{1}{c|}{Shen-poly}                                                                   & trig-poly                                                           \\ \hline
BCs                    & \multicolumn{3}{c|}{$--$}                                                                                                                  & \begin{tabular}[c]{@{}c@{}}periodic\\ (implicitly)\end{tabular} & \multicolumn{1}{c|}{sphere}                                                          & \multicolumn{1}{c|}{\begin{tabular}[c]{@{}c@{}}Dirichlet\\ Neumann\\ Robin\\ mixed\end{tabular}} & \begin{tabular}[c]{@{}c@{}}Dirichlet\\ Neumann\\ mixed\end{tabular} \\ \hline
Grids                  & \multicolumn{4}{c|}{uniform}                                                                                                                                                                               & \multicolumn{1}{c|}{\begin{tabular}[c]{@{}c@{}}spherical\\ coordinates\end{tabular}} & \multicolumn{1}{c|}{Gaussian}                                                                    & uniform                                                             \\ \hline
\end{tabular}
\end{center}
\end{table}

\subsection{SPFNO}
\begin{figure}[h!]
  \centerline{\includegraphics[width=0.9\textwidth]{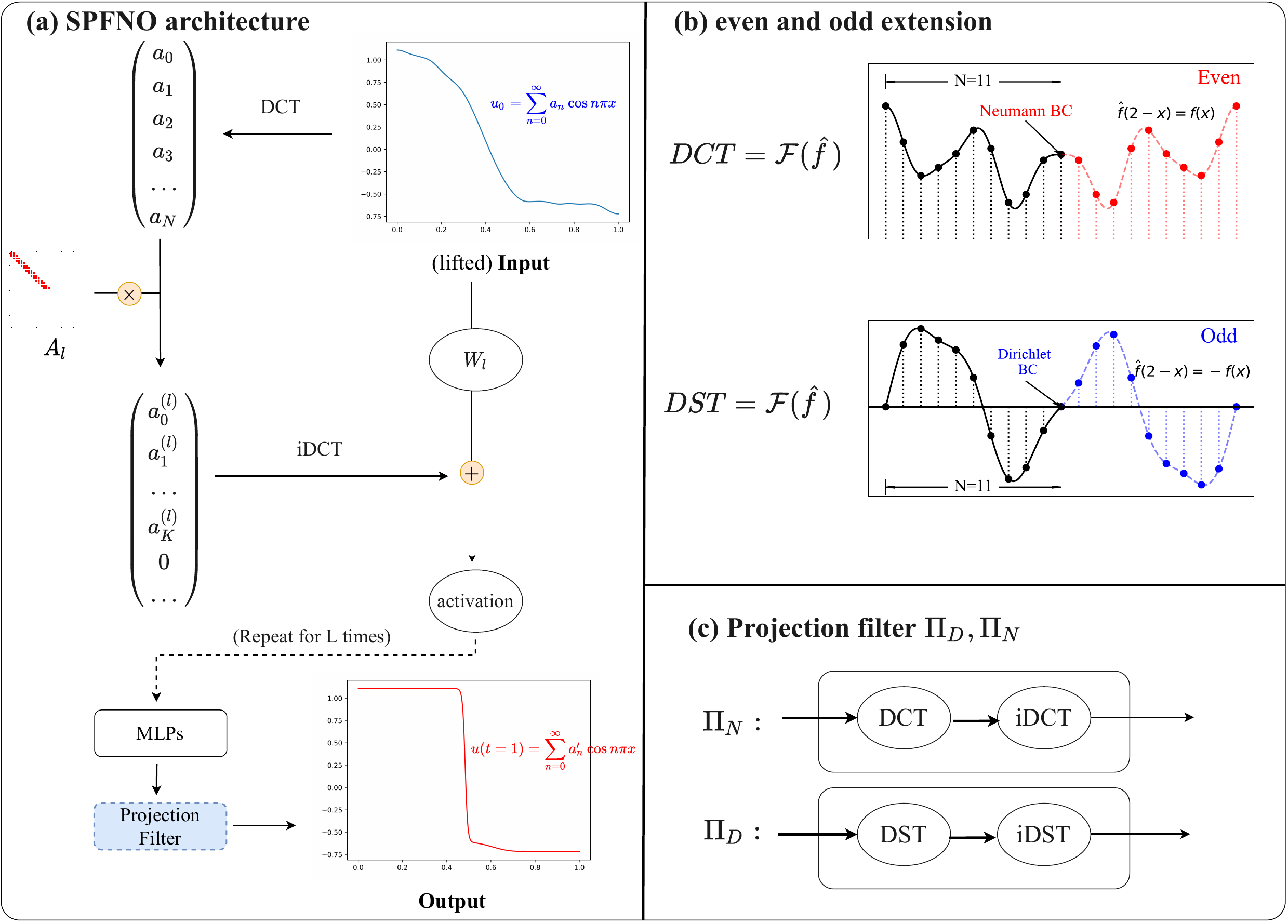}}
  \caption{(a) \textbf{Sketch map of SPFNO}, where $A_l$ is a learnable $b$-diagonal matrix with a truncation of $k_{max}$, where the $b$ represents the bandwidth. (b) \textbf{Diagram of odd/even extensions}. The Fourier expansion of odd/even extensions
    only contain the sine or cosine components, respectively. (c) The sketch map of the optional projection filter.} \label{fig:SPFNO}
\end{figure}
In this subsection, the domain $\Omega$ is limited to one dimensional interval $[0, 1]$ for convenience, while the conclusions can be easily generalized to separable multi-dimensional
domains.
When focusing on the cases of Dirichlet or Neumann BCs, the basis of specified trigonometric polynomials with ``semi-period'' is a competitive choice due to its fast transformation algorithms. More
concretely, for Dirichlet BCs, it represents the Fourier sine polynomials $u_k(x) = \sin k\pi
x, x \in [0, 1]$ and (discrete) sine transform; and the basis of cosine polynomials $u_k(x) = \cos k \pi x, x \in [0, 1]$ and (discrete) cosine transform for Neumann BCs.

The spectral methods of the aforementioned specified trigonometric polynomials have been investigated in \cite{bueno2014fourier,wise2021pseudospectral}. The SPFNO model is a natural derivative of these
methods under the framework of SOL, with $\mathcal T$ substituted by the sine or cosine transform.

The discrete sine/cosine transforms (DST/DCT) of the 1st kind involve taking the discrete Fourier transform on the odd/even extensions of the given input, respectively. With a slight abuse of notation, we denote $\hat f$ as the odd extension of function $f \in C[0, 1]$ if $f$ satisfies Dirichlet BCs or the even extension if
$f$ satisfies Neumann BCs:
\begin{equation}
\begin{split}
  \text{Odd:} \ \hat f(x) = \left\{\begin{aligned}
                     &f(x), \ x \in \left[ 0, 1 \right],\\
                     &-f(2-x), \ x \in \left[ 1, 2 \right];
                                   \end{aligned}\right. \quad
\text{Even:} \ \hat f(x) = \left\{\begin{aligned}
                     &f(x), \ x \in \left[ 0, 1 \right],\\
                     &f(2-x), \ x \in \left[ 1, 2 \right].
                                  \end{aligned}\right.
\end{split}\nonumber
\end{equation}
So the solutions are discretized at uniform grids and BCs are imposed the on boundary points, see Fig \ref{fig:SPFNO}b. Then, the sine or cosine polynomials form a basis for $\hat f$.
\begin{thm}\label{thm:1}
Suppose $f \in C[0, 1]$. The extended function $\hat f$ can be uniquely deconstructed by cosine polynomials $\left\{ \cos k \pi x \right\}_{k \in N}$ if and only if $f$ satisfies the Neumann BCs and can be
uniquely deconstructed by sine polynomials $\left\{ \sin k \pi x \right\}_{k \in N^+}$ if and only if $f$ satisfies the Dirichlet BCs.
\end{thm}

The proof of Theorem \ref{thm:1} will be presented in Section 3.2. Theorem \ref{thm:1} ensures the effectiveness of DST and DCT as decomposition transforms in solving PDEs with corresponding specific BCs. It also leads to the following conclusion:
\begin{cor}
The outputs of the SPFNO model strictly satisfy the Dirichlet and Neumann BCs, respectively. 
\end{cor}

Finally, the structure of the
SPFNO is given in Fig. \ref{fig:SPFNO}. The projection filter illustrated in Fig. \ref{fig:SPFNO}c is an optional component that ultimately eliminates the tiny errors of output with respect to the BCs. Generally, its influence on
performance is minimal. For example, we simultaneously present the results of SPFNO with and without it in Sec. \ref{sec:exam5}. And it is worth noting that a FNO with a similar DFT projection
filter satisfies the periodic BCs exactly.

\subsection{Proof for Theorem \ref{thm:1}}\label{app:proof-thm-1}
Since $f$ is typically a strong solution to PDEs in the implementation of SPFNO, it is reasonable to assume that $f$ possesses sufficient smoothness, such as belonging to $f \in H^1([0, 1])$. But for now we only need $f \in C[0,
1]$, or the differentiability on the boundary additionally for the following proof of Neumann BCs.
\subsubsection{Case of $f$ satisfying Dirichlet BCs and $\hat f$ being its odd extension}\label{sec:proof-dirichlet}
 First, we will prove that if the odd extension $\hat f$ can be uniquely deconstructed by sine polynomials, i.e.
\begin{equation}
  \hat f = \sum\limits_{k=1}^{\infty} b_k \sin k \pi x, \ x \in \left[ 0, 2 \right], \nonumber
\end{equation}
then $f$ satisfies the Dirichlet BCs:
$$ f(0) = f(1) = 0.$$
The proof is straightforward: substituting $x=0$ and $x=1$ give that
$$ f(0) = \hat f(0) = 0 = \hat f(1) = f(1).$$
This equality is derived from the fact that a countably infinite number of zeros sum up to zero. On the other hand, given $f \in C\left[ 0, 1 \right]$ as an arbitrary function that satisfies homogeneous Dirichlet BCs, i.e.,
\begin{equation}
\label{eq:30}
f(0) = f(1) = 0, \nonumber
\end{equation}
its odd extension is of the form
\begin{equation}
\label{eq:31}
\hat f(x) = \left\{\begin{aligned}
                     &f(x), \ x \in \left[ 0, 1 \right],\\
                     &-f(2-x), \ x \in \left[ 1, 2 \right],
\end{aligned}\right. \nonumber
\end{equation}
which means that $\hat f$ is a continuous function on a closed interval. Consequently, it can be inferred from the Weierstrass approximation theorem for trigonometric series that the Fourier
series of $\hat f$ converges to $\hat f$ uniformly, which is of the form
\begin{equation}
\label{eq:fourier-expansion}
\hat f(x) = \sum\limits_{m=0}^{\infty}  a_m \cos m\pi x + \sum\limits_{n=1}^{\infty}  b_n \sin n \pi x, \ x \in \left[ 0, 2 \right], 
\end{equation}
where
\begin{equation}\label{eq:bn}
\begin{split}
  &a_m = \int_{0}^2 \hat f(x) \cos m \pi x \ \rm{d}x , \ m \in \mathbb N, \\
  &b_n = \int_0^2 \hat f(x) \sin n \pi x \ \rm{d}x, \ n \in \mathbb N^{+}. 
\end{split}
\end{equation}
Then it yeilds that
\begin{equation}
\label{eq:lhs-rhs}
\hat f(x) - \sum\limits_{n=1}^{\infty} b_n \sin n\pi x = \sum\limits_{m=0}^{\infty} a_m \cos m\pi x , \ x \in \left[ 0, 2 \right].
\end{equation}
Noting that the left-hand side of Eq. (\ref{eq:lhs-rhs}) is an odd function while the right-hand side is even. The only possibility is that it remains zero constantly. As a result, 
\begin{equation}
\hat f(x) = \sum\limits_{n=1}^{\infty} b_n \sin n\pi x, \nonumber
\end{equation}
where $b_n$ is determined by Eq. (\ref{eq:bn}).

\subsubsection{Case of $f$ satisfying Neumann BCs and $\hat f$ being its even extension}

The proof of this part is actually analogous to Sec. \ref{sec:proof-dirichlet}. Assume that $f \in C[0, 1]$ is differentiable on $x=0$ and $1$. On the one hand, given that the even extension $\hat f$ can be uniquely deconstructed by cosine polynomials, i.e.
\begin{equation}
\label{eq:cos-extension}
  \hat f = \sum\limits_{k=0}^{\infty} a_k \cos k \pi x, \ x \in \left[ 0, 2 \right]. 
\end{equation}
Noting that $\hat f$ is an even expansion. Consequently, we have
\begin{equation}
\begin{split}
  f'(1) &= (\hat f)'(1) = \lim\limits_{h \to 0} \frac{\hat f(1+h) - \hat f(1-h)}{2 h} \\
               &= \lim\limits_{h \to 0} \frac{f(1-h) - f(1-h)}{2 h} = 0.
\end{split}\nonumber
\end{equation}
Moreover, since Eq. (\ref{eq:cos-extension}) holds, we can further extend $\hat f$ to $\mathbb R$ with a period of $2$, which is here denoted as $\tilde f$. The periodicity of $\tilde f$ yields that the
following equation holds on $x = 0$:
\begin{equation}
\begin{split}
  f'(0) &= (\tilde f)'(0) \\
        &= \lim\limits_{h \to 0} \frac{\tilde f(0+h) - \tilde f(0-h)}{2 h}
        = \lim\limits_{h \to 0} \frac{ f(0+h) - \hat f(2-h)}{2 h} \\
               &= \lim\limits_{h \to 0} \frac{f(h) - f(h)}{2 h}
                 = 0.
\end{split}\nonumber
\end{equation}
So the Neumann BCs are satisfied by $f$.

On the other hand, similar to Sec. \ref{sec:proof-dirichlet}, the Fourier series of $\hat f$ converges to $\hat f$ uniformly, so Eqns. (\ref{eq:fourier-expansion}) \& (\ref{eq:bn}) also hold for
the even extension $\hat f$ when $f$ satisfies Neumann BCs. 
Then it yields that
\begin{equation}
\hat f(x) - \sum\limits_{m=0}^{\infty} a_m \cos m\pi x =  \sum\limits_{n=1}^{\infty} b_n \sin n\pi x \equiv 0, \ x \in \left[ 0, 2 \right]. \nonumber
\end{equation}
As a result, 
\begin{equation}
\hat f(x) = \sum\limits_{m=0}^{\infty} a_m \cos m\pi x, \nonumber
\end{equation}
where $a_m$ is determined by Eq. (\ref{eq:bn}).

\begin{cor}
  The sets of specified trignometric functions $\left\{\cos k \pi x \right\}_{k \in N}$ and $\left\{\sin k \pi x \right\}_{k \in N^{+}}$  form the orthonormal bases for
  corresponding function spaces, namely, the even extension $\left\{ \hat f_{\mathrm{even}}| f \in C[0, 1], \ f'(0)=f'(1)=0 \right\}$ and odd extension $\left\{ \hat f_{\mathrm{odd}}| f \in C[0, 1], \ f(0)=f(1)=0 \right\}$, respectively.
\end{cor}
\begin{proof}
As Theorem \ref{thm:1} has shown that the specified trignometic functions form a basis for corresponding function space, all we need to do is to prove is the orthonormality of the normalized bases. A simple calculation yields
\begin{equation}
\begin{split}
&\int_0^{2} \sin k\pi x \sin m\pi x = -\frac{1}{2} \int_0^{2} \cos\pi (k+m)x - \cos\pi(k-m) x =  \delta_{k, m} , \\
&\int_0^{2} \cos k\pi x \cos m\pi x = \frac{1}{2} \int_0^{2} \cos\pi (k+m)x + \cos\pi (k-m) x =  \delta_{k, m}. \\
&\int_0^{2} \sin k\pi x \cos m\pi x = \frac{1}{2} \int_0^{2} \sin\pi (k+m)x + \sin\pi (k-m) x = 0.
\end{split}\nonumber
\end{equation}
\end{proof}

\section{Numerical experiments}
In order to verify the accuracy and efficiency of the SOL architecture and the importance of
BC-satisfying property for neural operators, we compare SPFNO with multiple popular but non-BC-satisfying architectures by solving the following four PDEs on publicly available datasets: (1) 1D and
(2) 2D Burgers' equations with
Neumann BCs; (3) Darcy flow problem with Dirichlet BCs; (4) the coupled reaction diffusion equations with Neumann BCs; and (5) 2D incompressible flow through a pipe. These five questions comprehensively
address the common tasks for operator learning that we are primarily concerned with:
\begin{itemize}
  \item single-step predictions for time-dependent PDEs;
  \item multi-step predictions for time-dependent PDEs;
  \item solving parametric PDEs with the variable coeffients as input;
  \item autoregressive learning for time-dependent PDEs;
  \item learning the solution to PDEs on general geometries.
\end{itemize}

The accuracy of the models is measured by the average relative $L^2$ norm error (also known as
the relative mean square error, or RMSE) between the
predicted solution and the reference solutions and the $L^{\infty}$ norm error on the corresponding BCs. The maximum $L^2$ norm error on the test dataset, which represents the
empirical worst performance and is crucial for assessing the credibility of models, is also taken into consideration. The baseline models are listed below.
\begin{enumerate}
  \item \textbf{FNO} \cite{li2021fourier} is a state-of-the-art neural operator for parametric PDEs, especially those involving periodic BCs. FNOs have many applications and they have achieved
    impressive accuracy in practical application due to its architecture being similar to that of the spectral method. In the experiments, zero-padding has
    been applied to the input of FNO to enhance the accuracy on non-periodic BCs.
\item \textbf{OPNO} \cite{liu2022render} is the first proposed SOL method for non-periodic BCs such as Dirichlet, Neumann, and Robin BCs. It provided the first numerical example that verifies the competitive
  accuracy of deep-learning-based surrogate model to the numerical method, with the relative errors reaching the order of $10^{-6}$. Notice that the OPNO will be tested if and only if the data on
  Chebyshev-Gauss-Lobatto grids are provided by the dataset.
 \item \textbf{U-Net} \cite{ronneberger2015u} is a popular autoencoding deep learning architecture that combines the convolutional and deconvolutional layers. It has been proven to be a powerful model for tackling image segmentation
   tasks. It is also used as a baseline model in the PDEBench datasets and demonstrates considerable accuracy in specific PDE tasks. 
 \item \textbf{LSM} \cite{wu2023LSM} is a cutting-edge transformer-based neural PDE solver that consists of an autoencoding backbone and innovative neural spectral blocks. It successfully trained a
   neural network with considerable depth and outperformed the performance of 14 existing models, including neural operators, autoencoders, and transformers, across 7 PDE-solving tasks. In this paper, the LSM model also serves as the most representative non-neural-operator model for PDEs known to us.
 \end{enumerate}
 
 \begin{table}[h!]
\caption{Hyperparameters configurations for the networks of SPFNO and the training process in the following experiments. The number of modes, depth of the spectral backbone, width of the
channel domain, and bandwidth of $A_l$ are denoted by $k_{max}$, $L$, $W$, and $b$, respectively.}\label{tab:para}
\begin{center}
\scriptsize
\begin{tabular}{@{}ccccccccc@{}}
\toprule
Experiment     & epochs & $L$ & $k_{max}$ & $W$ & $b$ & batch size & weight decay  & scheduler \\ \midrule
1D-Burgers     & 5000   & 4     & 20    & 50    & 4        & 20         & 1e-4 & StepLR\\
2D-Burgers     & 3000   & 4     & 16    & 24    & 4        & 20         & 1e-4 &StepLR\\
2D React-Diff  & 500    & 4     & 24    & 24    & 1        & 5         & 1e-4 & StepLR\\
Darcy flow     & 500    & 4     & 24    & 32    & 1        & 20         & 1e-6& Plateau\\
Corrected pipe & 500    & 4     & 24    & 32    & (3, 1)   & 20    & 1e-6 & Plateau \\ \bottomrule
\end{tabular}
\end{center}
\end{table}
 
In addition, the \textbf{BOON} \cite{saad2023guiding} model is an alternative and, to the best of our knowledge, the only other type of BC-satisfication technique distinct from SOL. It applies a corrective projection that enforces the BCs on each 
spectral layer of given neural operators, instead of using a basis that inherently satisfies the BCs. It also showed that the BC-satisfying correction significantly increase the accuracy of the
neural operators. In Sec. \ref{sec:addi-exp}, we selected three representative problems from BOON datasets \cite{saad2023guiding} as supplementary experiments: (6.a) 1D Burgers' equation with Dirichlet BCs, on which the BOON
showed the highest performance improvement (30X) over its baseline; (6.b) 1D heat equation with time-dependent BCs; and (6.c) 2D wave equation with Neumann BCs. Moreover, the selected datasets precisely
covers all the cases of model dimensions considered in \cite{saad2023guiding} (1D, 1D+time, 2D+time). In addition, the currently available portion of the BOON code allows us to include it in the comparison in Experiment 1.


Except for Sec. \ref{sec:to-fine}, all models are evaluated at the same resolution as that used during the training
process. Additionally, they are trained using the Adam optimizer and the RMSE loss function, with the same weight decay parameter and number of epochs, and
the random seeds are fixed to 0 for fairness. For neural operators, considering their similar spectral structures, we choose the same number of modes (denoted by $k_{max}$), depth of the spectral backbone ($L$), width of the
channel domain ($W$), and size of mini-batch. It is very common in spectral methods to compare the accuracy of various spectral bases by controlling the number of modes. A representative example is found in Boyd's work \cite{boyd1978choice}, where the number of modes needed to discretize a given function was used as a decisive criterion for comparing spectral methods with different bases, such as Chebyshev, Fourier, and spherical harmonics. The non-neural-operator baseline are also finely tuned, and only the best performance of baselines with comparable scale in our experiment is demonstrated. All experiments are performed
on an Nvidia A100 80GB GPU.

\subsection{Experiment 1: 1D viscous Burgers' equation with Neumann BCs}\label{exp:1d-burgers}
We first consider the one-dimensional viscous Burgers equation as follows
\begin{equation}\label{eq:burgers}
\frac{\partial}{\partial t} u(x, t) + u(x, t) \frac{\partial }{\partial x} u(x, t) = \nu \frac{\partial^2 }{\partial x^2}  u(x, t) , x \in \Omega \nonumber
  \end{equation}
subject to the initial-boundary conditions
\begin{align}
  & u(x, 0) = u_0(x), \ x \in \Omega, \nonumber \\ 
  & \frac{\partial}{\partial x}  u(x, t) = 0, \ x \in \partial \Omega, \nonumber
\end{align}
where $\Omega = \left[ -1, 1 \right]$. Burgers' equation is a fundamental PDE with applications in modeling turbulence, nonlinear acoustics, and
traffic flow. The complexity of the dynamical system it describes poses challenges for the learning of deep models, so it has been adopted as one of the most popular benchmark
problems in the field of
AI4Science.

The dataset used in this experiment is sourced from the paper on the OPNO \cite{liu2022render}, where $\nu$ is fixed to $0.1/\pi$ and the task is to learn the solution operator
$S_1: u_0(x) \to u(x, t=1)$. Because the input functions in PDE training and test dataset are usually randomly sampled from the same distribution, in order to ensure the generalization of the trained
model, the generator of samples should possess the ability to approximate any function on the input space and a sufficient number of degrees of freedom. In this dataset, the initial condition $u_0(x)$ is generated by sampling from a Gaussian random field
according to $u_0 \sim \mathcal N(0, 625(-4\Delta + 25 I)^{-2})$ with Neumann BCs to maintain the complexity of the sample space.

As we have concluded in \cite{liu2022render}, the BC-satisfying property is crucial for enhancing the accuracy and credibility of surrogate models for PDEs. One can reasonably anticipate that other BC-satisfying
neural operators should also demonstrate notable superiority, even when employed on different grids. This expectation arises from the fact that neural operators are specifically designed to learn the
underlying operator, rather than the values at discrete grids of a function.

Our results first ruled out the suspicion that the advantage of SPFNO over FNO arises from a larger bandwidth rather than BC satisfaction. This is illustrated in Fig. \ref{fig:burgers-agst-bw}, which also indicates that this technique is not suitable for FNO. Therefore, the bandwidth of FNO is maintained at 1.

The results presented in
Tab. \ref{tab:exp-1-1} and Fig. \ref{fig:burgers-agst-bw} show that the SPFNO achieves a comparable average accuracy and superior worst performance when compared to the previous SOL model OPNO. On the other hand,
it significantly outperforms all non-BC-satisfying baseline models as well as the BC-satisfying BOON model. And although the FNO and LSM models have achieved acceptable global errors in some cases, their errors on the BCs can be several orders of magnitude larger.

\begin{table}[h!]
  \caption{Evaluation of (a) the relative $L^2$ norm error ($\times 10^{-2}$) and worst error ($\times 10^{-2}$), and (b) the error on the Neumann BC of 1D-Burgers’ equation with resolution of
    $N$. The best result is in bold and the second best is underlined.  The SOL models, namely the SPFNO and OPNO that achieve the SOTA performance, are listed separately.}\label{tab:exp-1-1}
  \centering
\begin{subtable}{.61\textwidth}
	\scriptsize
  \caption{Relative $L^2$ norm error ($\times 10^{-2}$) and worst error ($\times 10^{-2}$)}
  \begin{tabular}{ccccccc}
    \toprule 
  \multirow{2}{*}{Model}   & \multicolumn{2}{c}{$N=256$} & \multicolumn{2}{c}{$N=1024$} & \multicolumn{2}{c}{$N=4096$} \\
    \cmidrule(lr){2-3} \cmidrule(lr){4-5} \cmidrule(lr){6-7}
   &$L^2$ &worst &$L^2$ &worst &$L^2$ & worst \\
    \midrule
  FNO
     &$1.57$& $14.9$ &$1.68$ &$14.1$ & $1.69 $ & $16.6$ \\
  U-Net
     &$6.27 $&$ 56.9 $ &$27.6$&$ 126.3 $ &$ 33.0$&$ 146.3 $ \\
  LSM
     &$4.87 $&$ 47.2 $ &$21.1$& $108.6  $ &$38.7$&$ 167.7 $ \\
  BOON
     &$1.20$&$ 10.1 $ &$1.28$& $ 10.2$ &$1.42$&$10.0$ \\ \midrule
  OPNO
     &$\mathbf{0.770}$&$\underline{4.63}$ &$\mathbf{0.781}$&$ \underline{4.40}$ & $\mathbf{0.782} $&$\underline{3.86}$ \\
  SPFNO
    &$\underline{0.868}$&$\mathbf{3.58} $ &$\underline{0.862}$&$ \mathbf{3.55}$ & $\underline{0.873}$&$ \mathbf{3.65}$ \\
    \bottomrule
  \end{tabular}
\end{subtable}
 \begin{subtable}{.33\textwidth}\label{tab:exp-1-1-b}
	\scriptsize
  \caption{Error on the BCs}
  \begin{tabular}{ccc}
    \toprule 
   $N=256$ & $N=1024$ & $N=4096$ \\
    \cmidrule(lr){1-1} \cmidrule(lr){2-2} \cmidrule(lr){3-3}
    $L^{\infty}_{BC}$  &$L^{\infty}_{BC}$ &$L^{\infty}_{BC}$  \\
    \midrule
  $2.9E\text{-}1$ & $4.1E\text{-}1$&$5.5E\text{-}1$ \\
  $1.1E\text{-}2$ &$5.7E\text{-}2$  & $8.9E\text{-}2$ \\
  $4.8E\text{-}1$ &$5.1E\text{-}1$  & $9.8E\text{-}1$ \\
  $0$ & $0$ &$0$ \\ \midrule 
  $6.0E\text{-}12$ & $1.1E\text{-}10$ &$1.9E\text{-}9$ \\
  $0$ & $0$ &$0$ \\
    \bottomrule
  \end{tabular}
\end{subtable}
\end{table}

Moreover, it is noteworthy that the BOON model also achieves better performance than any
non-BC-satisfying models. This further supports our viewpoints on the BC-satisfication of neural operators, as it is mutually corroborated through different neural architectures.

\begin{figure}[tbhp]
  \begin{center}
    \subfloat[][relative $L^2$ norm error]{\includegraphics[width=.35\linewidth]{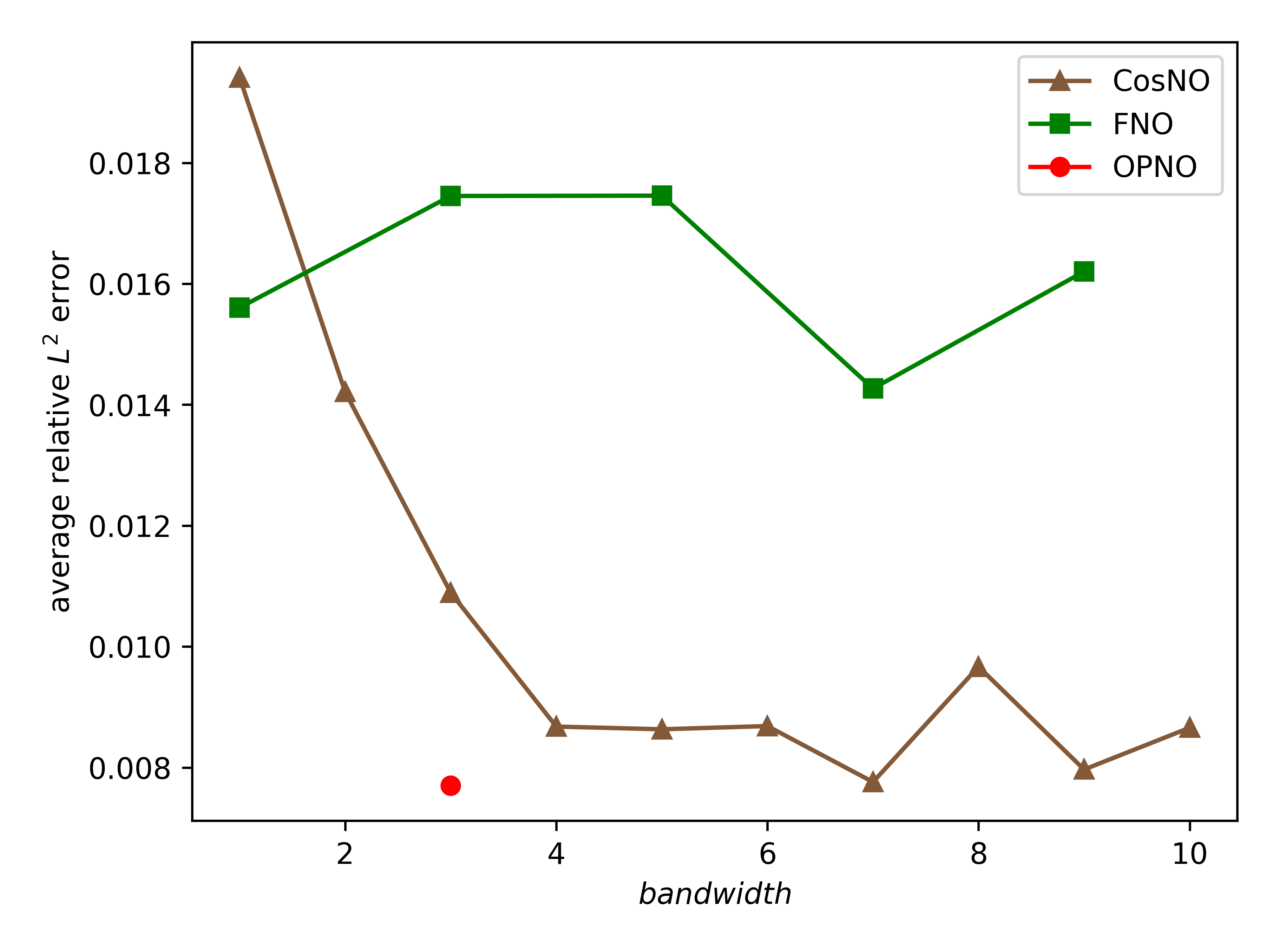}} \quad 
    \subfloat[][max error]{\includegraphics[width=.35\linewidth]{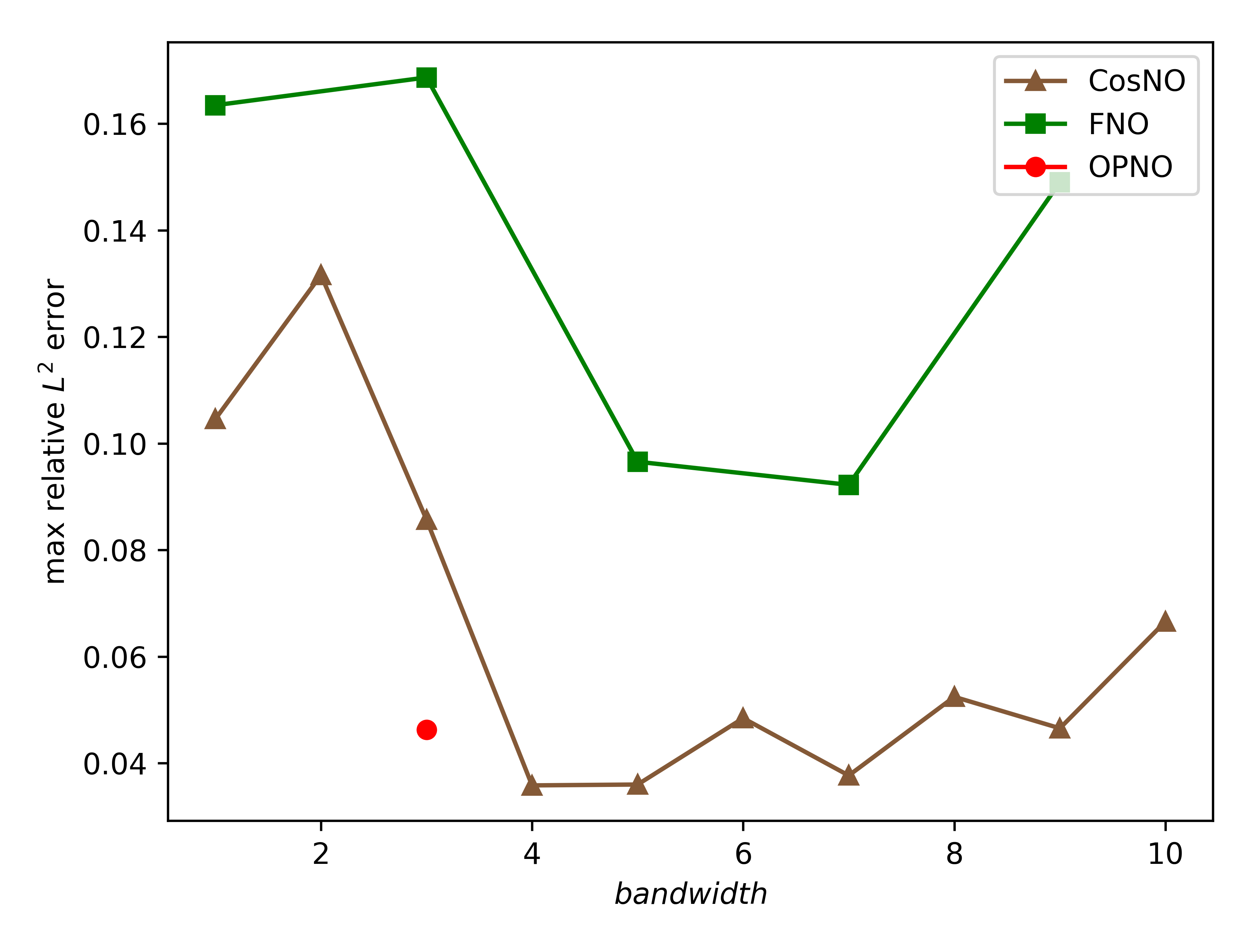}}
  \end{center}
    \caption{\textbf{Comparisons of the relative $L^2$  errors and max errors for various bandwidth $b$ on the testing dataset for 1D Burgers' equation}. While numerical experiment demonstate that the quasi-diagnolizing technique improves the performance
      of SPFNO, it cannot substitute the requirement of BC-satification property, seeing that the FNO with a $b$-diagonal learnable matrix does not result in a substantial accuracy improvement.}\label{fig:burgers-agst-bw}
\end{figure}

\subsection{Experiment 2: 2D viscous Burgers’ equation with Neumann BCs}\label{exp:2d-burgers}
With the computational domain set to $\Omega = [-1, 1]^2$, the following 2D Burgers' equation is considered in this experiment:
\begin{equation}
    \left\{\begin{aligned}
  & \frac{\partial u}{\partial t} + u \frac{\partial u}{\partial x} + u \frac{\partial u}{\partial y} = \nu (\frac{\partial^2 }{\partial x^2} +\frac{\partial^2 }{\partial y^2}) u,\\
  & \frac{\partial u}{\partial \mathbf n} (x, t) = 0, \ x \in \partial \Omega,
\end{aligned}\right.
  \nonumber
\end{equation}
Nevertheless, the output of the target operator consists of
solutions at multiple fixed time steps, i.e. $S_{\tau_0, ..., \tau_n}: u_0 \mapsto \left\{u(\cdot, \tau_1), ..., u(\cdot, \tau_n)\right\}$, so that the time-dependent PDEs can
be efficiently solved by only one forward propagation. As a result, both the operator and the task can be more complicated compared with the 1D case. We take the dataset from \cite{liu2022render} and choose a subset of time discretization by fixing
$\left\{ \tau_i \right\} = \left\{ 0.2, 0.6, 1.0 \right\}$. In this dataset, $u_0$ is generated according to the Gaussian random field $\mathcal N(0, 16(\Delta+16I)^{-2})$. During the training, the bandwidth of the learnable matrix $A_l$ is also set as $4$. The results are illustrated in
Tab. \ref{tab:exp-2-1}, where the results of different SOL models are also very close, and significantly outperform the baseline models.

\begin{table}[h!]
  \caption{Evaluation of the (a) relative $L^2$ norm error ($\times 10^{-2}$) and worst error ($\times 10^{-2}$), and (b) the error on the Neumann BCs of 2D-Burgers’ equation with resolution $N
    \times N$.}\label{tab:exp-2-1}
  \centering
\begin{subtable}{.61\textwidth}
	\scriptsize
  \caption{Relative $L^2$ norm error ($\times 10^{-2}$) and worst error ($\times 10^{-2}$)}
  \begin{tabular}{ccccccc}
    \toprule 
     \multirow{2}{*}{Model}    & \multicolumn{2}{c}{$N=50$} & \multicolumn{2}{c}{$N=100$} & \multicolumn{2}{c}{$N=200$} \\
    \cmidrule(lr){2-3} \cmidrule(lr){4-5} \cmidrule(lr){6-7}
    &$L^2$ &worst &$L^2$ &worst &$L^2$ & worst \\
    \midrule
  FNO
     &$0.528$& $9.02$ &$0.589$ &$10.03$ & $0.672 $ & $9.71$ \\
  U-Net
     &$1.64$&$16.0$ &$2.31$&$ 17.6 $ &$ 2.28$&$ 17.1 $ \\
  LSM
     &$2.43$&$ 9.23 $ &$2.93$&$ 12.4 $ &$ 3.19$&$ 11.6 $ \\ \midrule 
  OPNO
     &$\mathbf{0.371}$&$\mathbf{3.37}$ &$\mathbf{0.336}$&$ \mathbf{3.68}$ & $\mathbf{0.335} $&$\mathbf{3.68}$ \\
  SPFNO
    &$\underline{0.386}$&$\underline{4.98} $ &$\underline{0.378}$&$\underline{5.00}$ & $\underline{0.378}$&$ \underline{5.05}$ \\
    \bottomrule
  \end{tabular}
\end{subtable}
 \begin{subtable}{.33\textwidth}
	\scriptsize
  \caption{Error on the BCs}
  \begin{tabular}{ccc}
    \toprule 
   $N=50$ & $N=100$ & $N=200$ \\
    \cmidrule(lr){1-1} \cmidrule(lr){2-2} \cmidrule(lr){3-3}
    $L^{\infty}_{BC}$  &$L^{\infty}_{BC}$ &$L^{\infty}_{BC}$  \\
    \midrule
  $1.6E\text{-}1$ & $3.6E\text{-}1$&$7.8E\text{-}1$ \\
  $4.2E\text{-}2$ &$1.3E\text{+}0$  & $4.2E\text{+}0$ \\
  $1.9E\text{-}1$ &$4.1E\text{-}1$  & $1.6E\text{+}0$ \\ \midrule 
  $2.9E\text{-}12$ & $2.0E\text{-}12$ &$7.9E\text{-}12$ \\
  $0$ & $0$ &$0$ \\
    \bottomrule
  \end{tabular}
\end{subtable}
\end{table}

\subsection{Experiment 3: Coupled 2D Reaction--Diffusion Equations with Neumann BCs}
The coupled reaction--diffusion (Allen Cahn) equations are formulated as follows
\begin{equation}
\left\{\begin{aligned}
  \frac{\partial u}{\partial t} &= d_u \frac{\partial^2 u}{\partial x^2} + d_u \frac{\partial^2 u}{\partial y^2} + R_u(u, v),\\
\frac{\partial v}{\partial t} &= d_v \frac{\partial^2 v}{\partial x^2} + d_v \frac{\partial^2 v}{\partial y^2} + R_v(u, v), 
\end{aligned}\right.\nonumber
\end{equation}
where
\begin{equation}
\begin{split}
  &R_u(u, v) = u - u^3 - k - v, \\
  &R_v(u, v) = u - v, \\
  &d_u = 0.001, d_v = k = 0.005.
\end{split}\nonumber
\end{equation}
The nonlinearly coupled variables $u$ and $v$ represent the activator and inhibitor in the system, respectively, to which the Neumann BCs are imposed. The dataset is provided by PDEBench
\cite{ta2022pdebench}, a comprehensive set of benchmarks for scientific machine learning. Since
the data are given on a staggered uniform grid, directly sub-sampling would yields an unexpected non-uniform grid. So we only perform the experiment with the original resolution of $128\times 128$.

In \cite{ta2022pdebench}, models are trained and evaluated using an autoregressive approach, where the output at each time step serves as the input for the subsequent time step in a time series. This approach is often applied in predicting time-dependent PDEs \cite{li2021fourier} and global
weather \cite{pathak2022fourcastnet} using neural
operators but may lead to training instability and GPU memory limitations. The results are shown in Tab. \ref{tab:exp-3}, where SPFNO acheives the lowest errors.

Furthermore, it should be noted that the difference in efficiency between SPFNO and FNO is mainly due to the constant factor arising from performing FFT on the odd/even extensions of length 2N for the DST/DCT. However, by employing optimized butterfly operations, it is possible to achieve the same efficiency as FFT.

\begin{table}[h!]
  \caption{Relative $L^2$ norm error ($\times 10^{-2}$) and worst error ($\times 10^{-2}$) of (a) 2D reaction diffusion equations in Experiment 3 and (b) 2D Darcy flow problem in Experiment 4. The
    number of parameters in neural networks (\#Param) and the averaged GPU time per epoch during the training process (\#Time) are also given.}
  \centering
\begin{subtable}{.49\textwidth}
	\scriptsize
  \caption{Reaction diffusion equations}\label{tab:exp-3}
  \begin{tabular}{ccccc}
    \toprule 
    & $L^2$ & worst & \#Param & \#Time \\
    \midrule
 FNO   & $\underline{5.19}$ &$\underline{6.37}$ &$1.3$m & $150s$ \\
 U-Net   & $68.9$ &$77.6$ & $7.8$m&$217s$ \\
 LSM   & $7.20$ &$13.7$ &$1.2$m &$606s$ \\
    SPFNO   & $\mathbf{1.13}$ &$\mathbf{1.60}$ &$1.4$m  & $275s$ \\
    \bottomrule
  \end{tabular}
\end{subtable}
 \begin{subtable}{.49\textwidth}
	\scriptsize
  \caption{Darcy flow problem}\label{tab:exp-4}
  \begin{tabular}{ccccc}
    \toprule 
    & $L^2$ & worst & \#Param & \#Time \\
    \midrule
 FNO   & $0.688$ &$6.04$ &$2.4$m & $7.44s$ \\
 U-Net   & $0.989$ &$\underline{4.81}$ & $7.8$m&$6.53s$ \\
 LSM   & $\underline{0.468}$ &$\mathbf{2.72}$ &$19.2$m &$62.7s$ \\
    SPFNO   & $\mathbf{0.283}$ &$4.89$ &$2.4$m  & $32.9s$ \\
    \bottomrule
  \end{tabular}
\end{subtable}
\end{table}


\subsection{Experiment 4: 2-D Darcy flow with Dirichlet BCs}
Darcy’s law describes the flow of fluid through a porous medium. It has been widely implemented in various fields, including hydrogeology, petroleum engineering, and soil mechanics. In this
experiment, the 2-D steady-state Darcy
flow equations in a unit box are formulated as the following boundary value problem (BVP):
\begin{equation}\label{eq:darcy}
- \nabla \cdot (a(\mathbf x) \nabla u(\mathbf x)) = f, \ \mathbf x \in [0, 1]^2.
\end{equation}
Moreover, the homogeneous Dirichlet BCs are imposed. The task is to learn the operator $\mathcal G(a) = u$ that maps the diffusion coefficient $a$ to the solution $u$, where the input $a$ is possibly discontinuous. This problem serves as another most
commonly used benchmark for deep PDE solvers since the dataset is provided in \cite{li2021fourier}. In this dataset, the diffusion coefficient $a(x)$ is taken as a piecewise constant, while the
reference solution is generated using a finite difference method. Under this premise, however, the nondifferentiable
variable coefficient makes the 2nd order finite difference method unsuitable for solving the problem. So we utilize the 2D-Darcy dataset of PDEBench \cite{ta2022pdebench} with $f$ fixed as $100.0$ 
instead. Its input functions are also piecewise constant and its larger size ($10^4$ pieces of data compared to
$10^3$ in \cite{li2021fourier})
contributes to the model achieving higher accuracy. The experiment is conducted using the original spatial resolution $128 \times 128$.

Compared to the Neumann BC, the Dirichlet BC is much easier to learn because it does not involve any derivative. Additionally, the heterogeneity between the input and output functions leads to a much more
complicated spectral structure of the mapping operator. Actually, multiple non-neural-operator method have been reported to surpass the performance of neural operators in solving Eq. (\ref{eq:darcy}), especially the transformers,
among which the LSM acheives the highest accuracy that is known to us.

We adopt the ReduceLROnPlateau scheduler for SPFNO to accelerate the training. Similar scheduler is employed for baseline models if there is an improvement in accuracy. The results can be found in Tab. \ref{tab:exp-4}, where the
SPFNO again obtains the lowest relative error. Notably, however, the LSM achieves the lowest maximum error. The main reason is that the input $a(x)$ in the dataset is discontinuous, and when the
discontinuous input exhibits rapid variations, the accuracy of spectral methods will significantly decrease. Similarly, the worst performances of FNO and SPFNO are slightly inferior to
that of U-Net. It might help address such issues to introduce local structures such as convolution layers or radial basis functions to capture the extremely high frequency information, but could potentially
lead to a loss of flexibility or invariance in discretization \cite{xiong2023koopman}.

\subsection{Experiment 5: 2D Navier-Stokes equation with mixed BCs}\label{sec:exam5}
We focus on the ability and flexibility of handling general geometries of our model. The incompressible flow through a pipe is considered in this experiment, of which the governing equation is incompressible Navier-Stokes equation
\begin{equation}
\left\{\begin{aligned}
  & \frac{\partial \mathbf u}{\partial t} + (\mathbf u \cdot \nabla)\mathbf u = -\nabla p + \mu \nabla^2 \mathbf u, \\
  & \nabla \cdot \mathbf u = 0,
\end{aligned}\right.\nonumber
\end{equation}
where $u$ is the velocity vector, $p$ is the pressure, and  $\mu$ represents the viscosity. The Pipe dataset is given by the geometry-aware Fourier neural operator (Geo-FNO) \cite{li2022fourier}, where the shape of the pipe is randomly
generated, and the maximum velocity condition $u_{max} = [1; 0]$ is imposed at the inlet, free boundary condition is imposed
at the outlet, and no-slip boundary condition is imposed at the inner wall of the pipe. However, we have identified a few instances with artificial discontinuities at the upper edge of the outlet in
the current Pipe dataset, see Fig. \ref{fig:new-pipe}a. These discontinuities arise from bugs in the numerical solver when handling the free boundary condition. So we recomputed the reference solutions to address this
issue. The updated dataset is now available at \url{https://drive.google.com/drive/folders/1WPAs6bXttCPOWrDaUudC8B4dKPoju1OO}. Notably, after excluding the errors caused by the dataset, the accuracy of all models has improved on the updated dataset.

\begin{figure}[htb]
  \begin{center}
    \subfloat[][Updated pipe dataset]{\includegraphics[width=.49\linewidth]{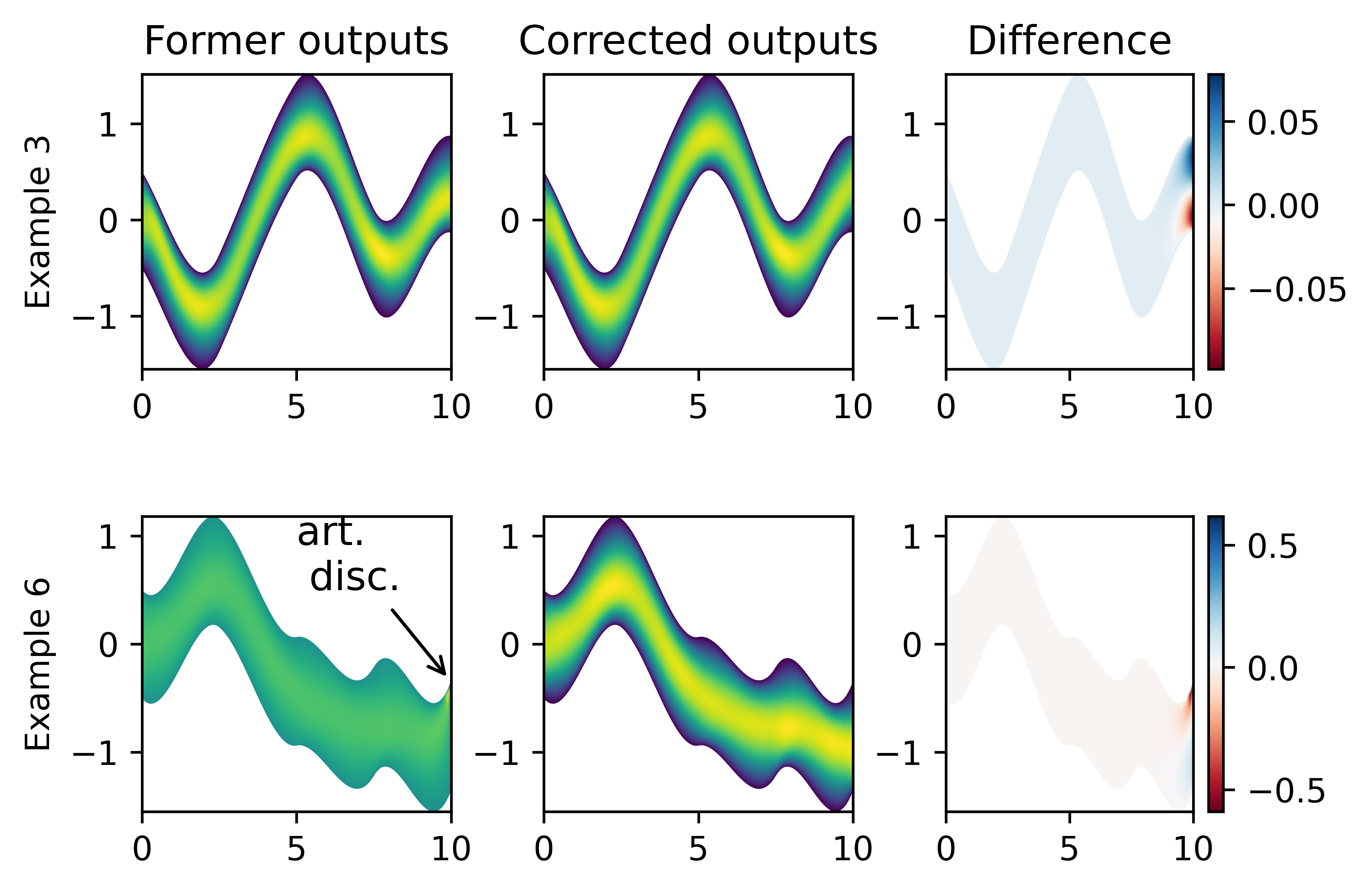}}
    \subfloat[][Dual Fourier extension]{\includegraphics[width=.49\linewidth]{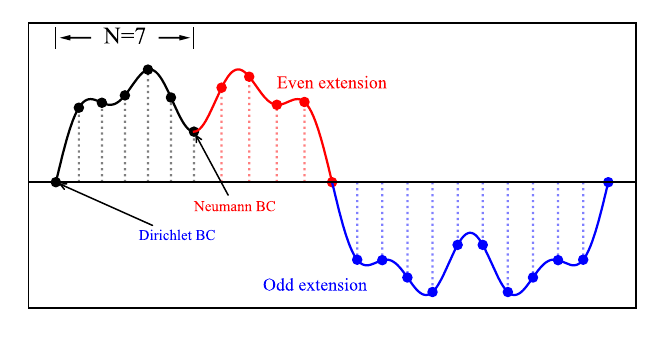}}
  \end{center}
   \caption{(a) Examples of the differences between updated pipe dataset and the previous version. (b) Diagram of the dual Fourier extension.}\label{fig:new-pipe}
\end{figure}

In the experiment, we utilized $1000$ instances for training and $1000$ instances for test. The goal is to learn the operator that maps the mesh point locations to the horizonal fluid velocity on
these mesh points. The experiment is conducted with a resolution of $129 \times 129$.

The solution to the problem satisfies a fixed Dirichlet BC on the inlet and a Neumann BC on the outlet. To approximate the mixed BCs in the horizontal direction, the Fourier transform is applied on the \textit{dual Fourier extensions}, as shown in Fig. \ref{fig:new-pipe}b. When a continuous function $f \in C[0, 1]$ satisfies $f(0)=0$ and $f'(1)=0$, the Fourier expansion of its odd-even dual Fourier
extension consists solely of odd-order sine series terms. The corresponding transform is also known as the ``WAWS'' transform \cite{wise2021pseudospectral}, which is an abbreviation for the Whole-sample Antisymmetry on the left boundary
and Whole-sample Symmetry on the right boundary. On the vertical direction, we use SPFNO with sine basis to impose the Dirichlet BCs. Finally, in this experiment, we implement a WAWS transformation in
the horizontal direction for the mixed BCs, and a sine transform in the vertical direction for Dirichlet BCs.

In addition to the mesh deviating from the expected uniform grid in the design of SPFNO, the complexity of the problem is further compounded by the non-perpendicular alignment of the pipe's artificial
outlet with the longitudinal axis. This could lead to errors in the Neumann BC we imposed on the mesh at the outlet. To improve the flexibility of our model when dealing with general geometries, the
projection filter is dropped. As shown in Tab. \ref{tab:exp-5}, the SPFNO models has a significant lower error compared to the baseline models. Besides, the projection filter only exhibits
minimal impact on the performance. These results demonstrate the flexibility of SPFNO in handling general geometries.

\begin{table}[h!]
\captionsetup{skip=1pt}
\caption{Evaluation relative $L^2$ norm error ($\times 10^{-3}$) and worst error ($\times 10^{-3}$) for the pipe flow.}\label{tab:exp-5}
\begin{center}
  \scriptsize
\begin{tabular}{@{}ccccc@{}}
\toprule
                                                          & $L^2$ error    & worst         & \#Param & \#Time(s) \\ \midrule
Geo-FNO                                                   & 1.78           & 6.31          & 2.4m    & 2.6       \\
U-Net                                                     & 4.68           & 9.35          & 7.8m    & 1.4       \\
LSM                                                       & \underline{1.69}     & \underline{4.08}    & 10.8m   & 6.7       \\ \midrule
SPFNO                                                     & \textbf{0.980} & \textbf{3.56} & 7.2m    & 5.8       \\
SPFNO(with PF) & 0.999          & 3.83          & 7.2m    & 6.0         \\ \bottomrule
\end{tabular}
  
\end{center}
\end{table}

\subsection{Additional experiements}\label{sec:addi-exp}

\subsubsection{Evaluation and comparison on the dataset of BOON}
We additionally compare our model with the BOON model on the Dirichlet and Neumann datasets that the latter provided. Similar parameters are chosen as those for BOON in \cite{saad2023guiding} (also
see Tab. \ref{tab:boon-para}), and the bandwidth
$b$ is fixed as 1. Please note that one complex Fourier mode consists of two basis functions.

\begin{table}[h!]
 \captionsetup{skip=0pt}
  \caption{Parameters of the SPFNO model for the training process on the dataset of BOON}\label{tab:boon-para}
  \begin{center}
 \scriptsize
\begin{tabular}{@{}cccc@{}}
\toprule
Problem and the dimension        & Modes & Channels & \#Param \\ \midrule
1D Burgers' equation (1D)        & 32    & 64    & 0.55m   \\
1D heat equation (1D space+time) & 24    & 32    & 0.11m   \\
2D wave equation (2D space+time) & 16    & 20    & 0.43m   \\ \bottomrule
\end{tabular}
\end{center}
\end{table}

To accelerate the training process, a ReduceLROnPlateau scheduler is adopted as in
Experiment 4 for the SPFNO model.

\textit{(a) 1D Burgers' equation with Dirichlet BCs}: In this dataset, the following Riemann problem of viscous Burgers' equation with Dirichlet BCs is consider:
\begin{equation}
\begin{split}
  &u_t + (u^2/2)_x = \nu u_{xx}, \ x \in [0, 1], t \geq 0,\\
  &u_0(x) = \left\{\begin{aligned}
                    u_L, \ \textrm{if} \ x \leq 0.5,\\
                     u_R, \ \textrm{if} \ x > 0.5, 
                   \end{aligned}\right.\\
  &u(0, t) = u_{\textrm{exact}}(0, t), \ u(1, t) = u_{\textrm{exact}}(1, t), \ t > 0.
\end{split}  \nonumber
\end{equation}
The task is to learn the solution operator $\mathcal S: u_0(x) \to u_{t=1.2}$.

\textit{(b) 1D heat transfer equation with Neumann BCs}: In this dataset, the following 1D heat equation with time-dependent Neumann BCs is consider:
\begin{equation}
\begin{split}
  & u_{t} - k u_{xx} = f(x, t), \ x \in [0, 1], \ t \geq 0,  \\  
  & u_0(x) = \cos (\omega\pi x), \ x \in [0, 1], \\
  & u_x(0, t) = 0, \ u_x(1, t) = U \sin \pi t, \ t \geq 0,
\end{split}\nonumber
\end{equation}
The task is to learn the solution operator $\mathcal S: u_0(x) \to \left\{ u(x, t=t_M) \right\}$ where $M = 25$. 

\textit{(c) 2D wave equation with Neumann BCs}: In this dataset, the following 2D wave equation with Neumann BCs is consider:
\begin{equation}
  u_{tt} = c^2(u_{xx} + u_{yy}), \ x, y \in [0, 1]^2, \ t \geq 0,    \nonumber
\end{equation}
to which the following analytical solution exists
$$ u_{\textrm{exact}}(x, t) = k \cos(\pi x) \cos(\pi y) \cos(c \sqrt{2} \pi t). $$
The task is to learn the solution operator $\mathcal S: u_0(x, y) \to \left\{ u(x, y, t=t_M) \right\}$ where $M = 25$.

\begin{table}[h!]
 \captionsetup{skip=0pt}
  \caption{Performance of SPFNO and BOON models on multiple BOON datasets. The relative $L^2$ test error (and error on BCs within the parentheses) are given. The averaged GPU time per epoch in the
    training process is given within square brackets or seperately.}\label{tab:boon-perform}
  \setlength{\tabcolsep}{2pt}
\begin{center}
 \scriptsize
 \begin{subtable}{1\textwidth}
 \begin{center}
  \caption{Single-step predictions for 1D Burgers’ equation with Dirichlet BCs and varying viscosities $\nu$. $N=500$. }\label{tab:exp-6-a}
  \begin{tabular}{ccccccc}
    \toprule 
   Model & $\nu=0.1$ & $\nu=0.05$ & $\nu=0.02$ & $\nu=0.005$ & $\nu=0.002$ & \#Time\\
    \midrule
    SPFNO& $\mathbf{1.12e-5}(0)$ &$\mathbf{1.26e-5}(0)$ &$\mathbf{5.95e-5}(0)$  & $5.04e-4(0)$ & $7.28e-4(0)$ & $0.12s$\\
    BOON-FNO \cite{saad2023guiding}& $1.2e-4(0)$ &$1.0e-4(0)$ &$8.4e-5(0)$  & $\mathbf{1.0e-4}(0)$ & $1.27e-3(0)$ & $1.3s$\\
    BOON-MWT \cite{saad2023guiding}& $2.0e-4(0)$ &$2.5e-4(0)$ &$2.2e-4(0)$  & $2.0e-4(0)$ & $\mathbf{3.4e-4}(0)$ & $--$\\
    \bottomrule
  \end{tabular}
\end{center}
\end{subtable}\vspace{10pt}
\begin{subtable}{1\textwidth}
\begin{center}
  \caption{Multi-step predictions for 1D heat equation with Neumann BCs and varying resolutions N. $M = 25$. }\label{tab:exp-6-b}
  \begin{tabular}{ccccc}
    \toprule 
    Model& $N=256$ & $N=512$ & $N=1024$ & $N=2048$ \\
    \midrule
    SPFNO&$\mathbf{5.39e-4}(0) [0.51s]$ &$\mathbf{3.53e-4}(0) [0.90s]$  & $\mathbf{2.63e-4}(0) [0.96s]$ & $\mathbf{2.90e-4}(0) [1.13s]$ \\
    BOON  &$3.18e-2(0) [3.5s]$ &$3.67e-2(0) [5.9s]$ & $4.24e-2(0) [10.4s]$ & $6.76e-3(0) [20.0s]$ \\
    \bottomrule
  \end{tabular}
\end{center}
\end{subtable}\vspace{10pt}
\begin{subtable}{.99\textwidth}
\begin{center}
  \caption{Multi-step predictions for 2D wave equation with Neumann BCs and varying resolutions N. $M = 25$.}\label{tab:exp-6-c}
  \begin{tabular}{cccc}
    \toprule 
    Model& $N=25$ & $N=50$ & $N=100$ \\
    \midrule
    SPFNO &$\mathbf{1.14e-4}$(0)  & $\mathbf{4.69e-5}$(0) & $\mathbf{1.93e-4}$(0) \\
    BOON \cite{saad2023guiding} &$9.7e-4$(0) & $8.93e-4$(0) & $9.6e-4$(0) \\
    \bottomrule
  \end{tabular}
\end{center}
\end{subtable}
\end{center}
\end{table}

The results are illstrated in Tab. \ref{tab:boon-perform}, where the SPFNOs outperform BOONs in almost all of the cases, and we have observed that continuing the training process allows the SPFNO
surpassing the baseline in all scenarios. However, this process would not be carried out, and the reasons will be given in the subsequent text. Due to the adoption of a global spectral method in SPFNO,
in contrast to BOON, which empolys a local finite difference method for the approximation of BCs, the SPFNO achieves higher accuracy
globally by introducing the physical information from boundaries into the interior system in a direct way.

Furthermore, the SPFNO also demonstrates approximately an order of magnitude higher efficiency compared to BOON due to its utilization of bases that inherently satisfy the BCs, while BOON requires an additional BC
correction operation in each spectral layer. 

It is worth noting that, however, we suggest
that the primary factor contributing to the extremely high accuracy of our models in additional experiments is the dataset. The problems considered here have analytical solutions, and the generator of input
functions involves only a small
number of degrees of freedom. This circumstance enables the model to easily fit a manifold of significantly reduced dimensionality while learning the solution operator. Hence, although the existing results
sufficiently illustrate the approximation capability of all tested model for the problems, we are afraid that the practical implications of further error reduction remain limited.

\subsubsection{Invariance to discretization: the out-of-domain evaluations of the pre-trained model on different grids}\label{sec:to-fine}

We take the pre-trained model in Experiment 1 as an example. As a consequence of the spectral structure of the SOL architecture, the model trained on the coarse grid ($N=256$) can directly predict on the fine grid without significant losses in accuracy, see Tab \ref{tab:eval-fine} and
Fig. \ref{fig:eval-fine}. The data from the
sub-scaled grid remains untouched during the training process. A similar situation is also observed for the model trained on the fine grid ($N=4096$) but evaluated on coarse ones.

\begin{table}[h!]
  \caption{The evaluation relative $L^2$ errors (errors on BCs) of SPFNO model for 1D Burgers' equation. Model is trained on a $N$ grid but evaluated on a grid with resolution $N'$} \label{tab:eval-fine}
\begin{center}
\scriptsize
  \begin{tabular}{cccccc}
    \toprule 
    & $N'=256$ & $N'=512$ & $N'=1024$ & $N'=2048$ & $N'=4096$ \\
    \midrule
    $N=256$ &$0.0086766(0)$  & $0.0086748(0)$ & $0.0086737({0})$ &$0.0086732(0)$ &$0.0086729(0)$ \\
    $N=4096$ &$0.0087394(0)$  & $0.0087364({0})$ & $0.0087351({0})$ &$0.0087344({0})$ &$0.0087341({0})$ \\
    \bottomrule
  \end{tabular}
\end{center}
\end{table}

\begin{figure}[tbhp]
  \begin{center}
    \subfloat[][$N'=256$]{\includegraphics[width=.19\linewidth]{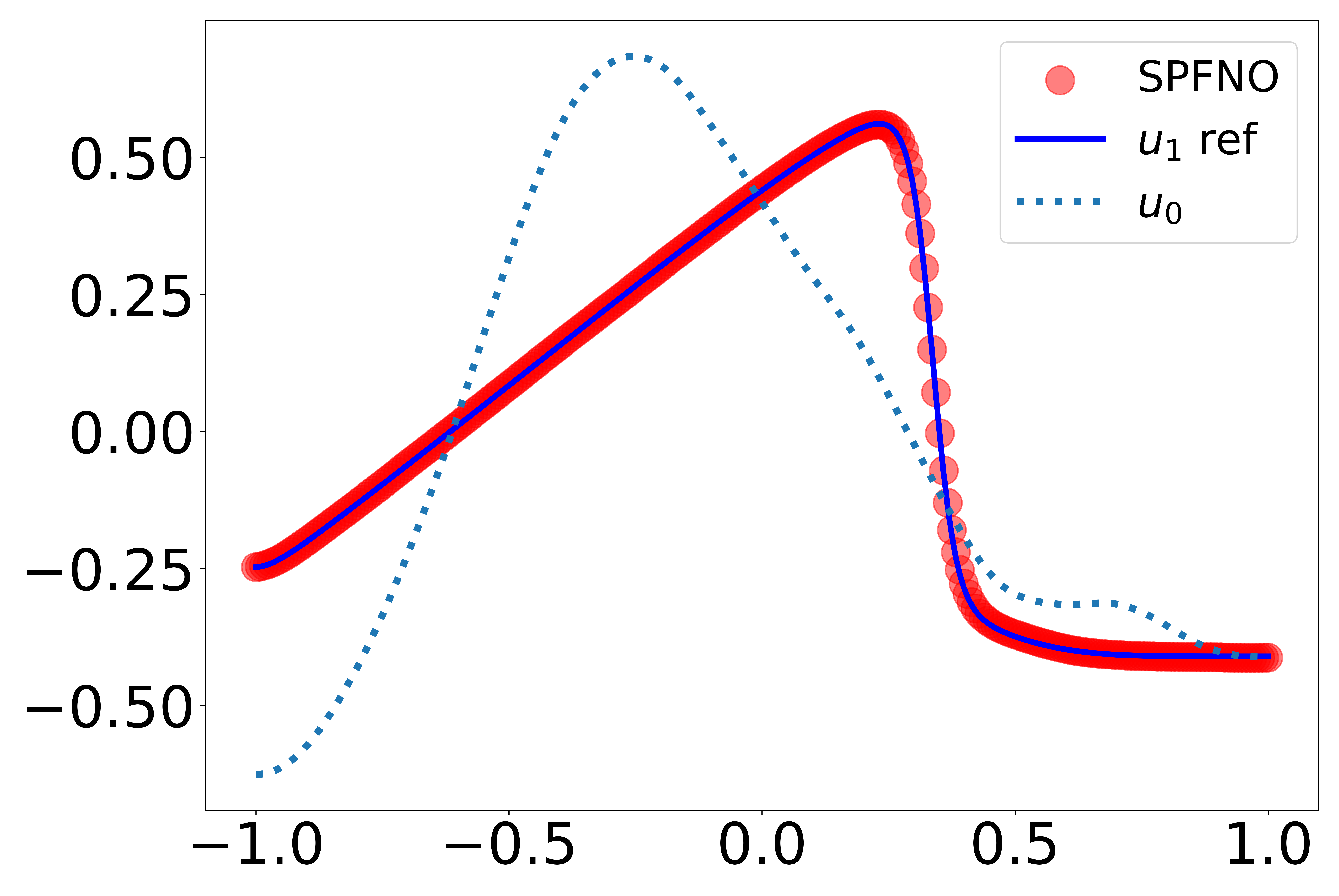}}
    \subfloat[][$N'=512$]{\includegraphics[width=.19\linewidth]{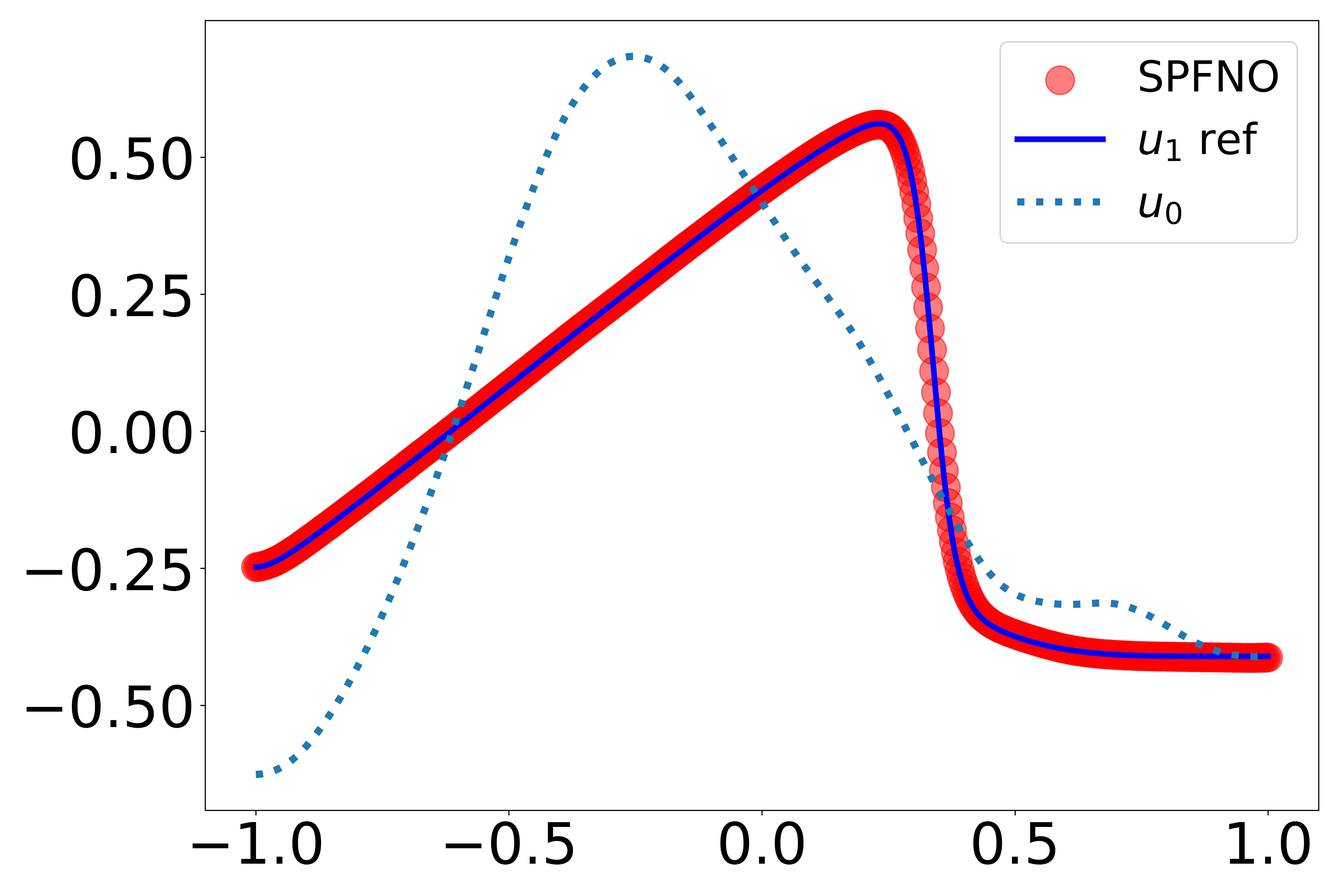}}
    \subfloat[][$N'=1024$]{\includegraphics[width=.19\linewidth]{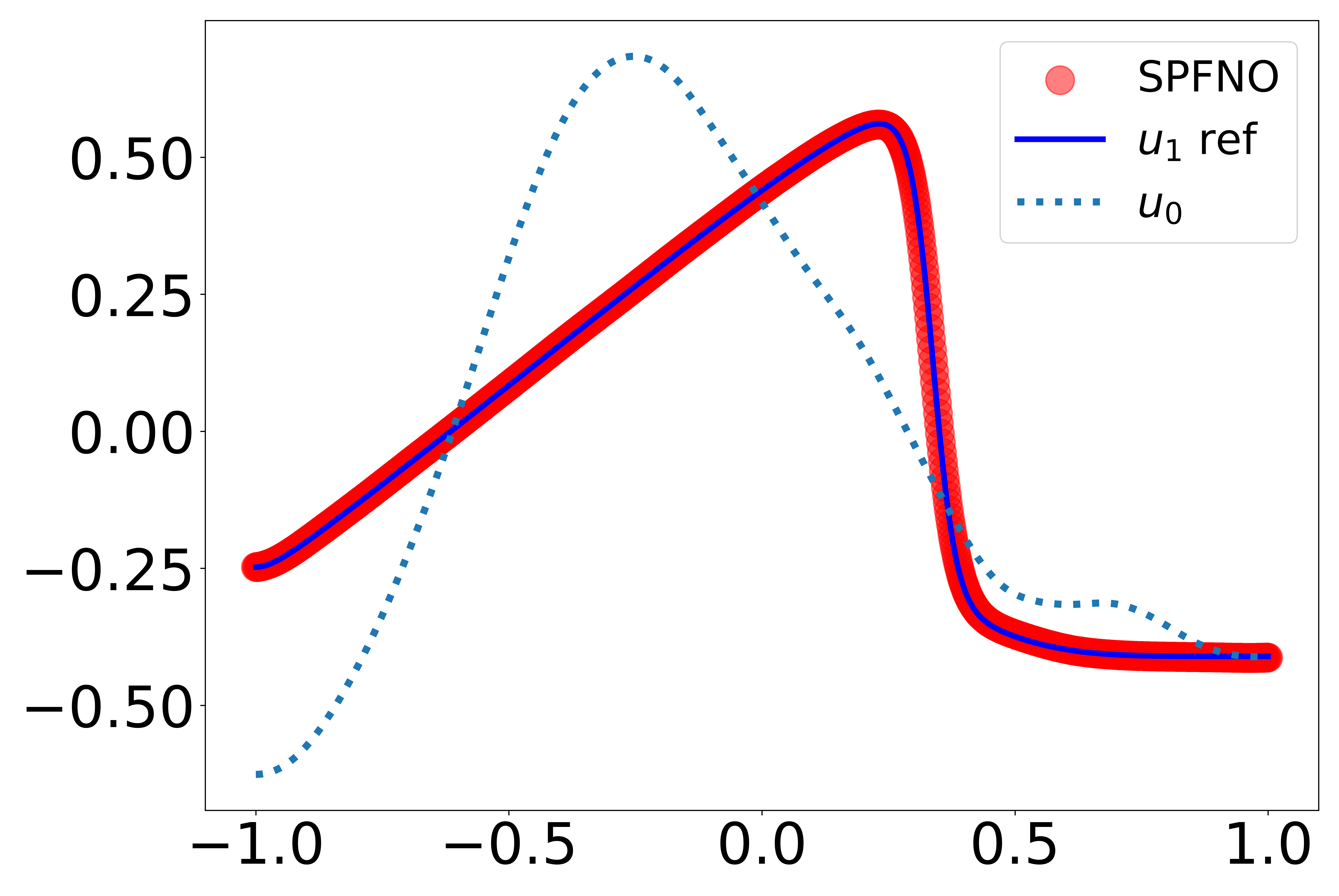}}
    \subfloat[][$N'=2048$]{\includegraphics[width=.19\linewidth]{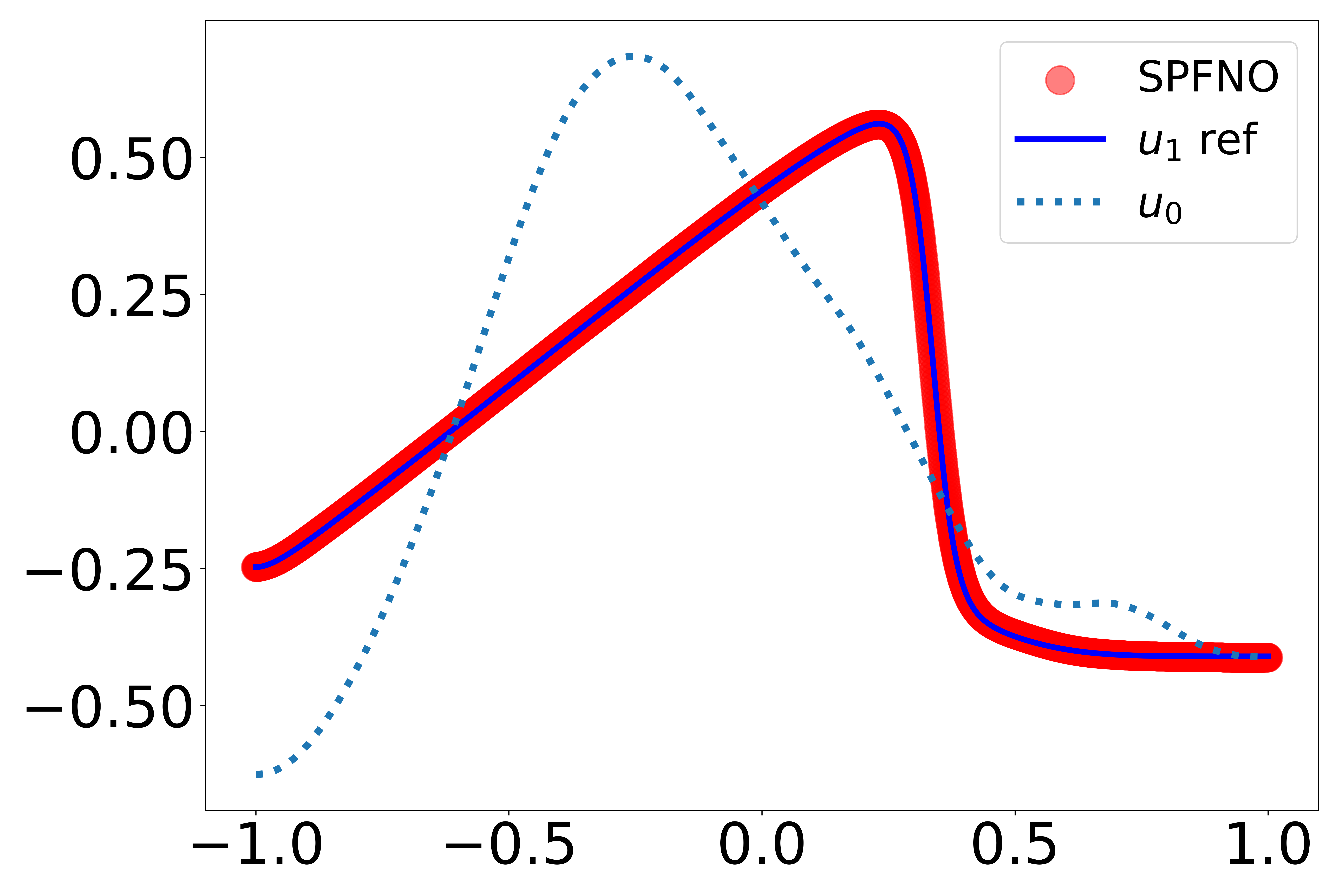}}
    \subfloat[][$N'=4096$]{\includegraphics[width=.19\linewidth]{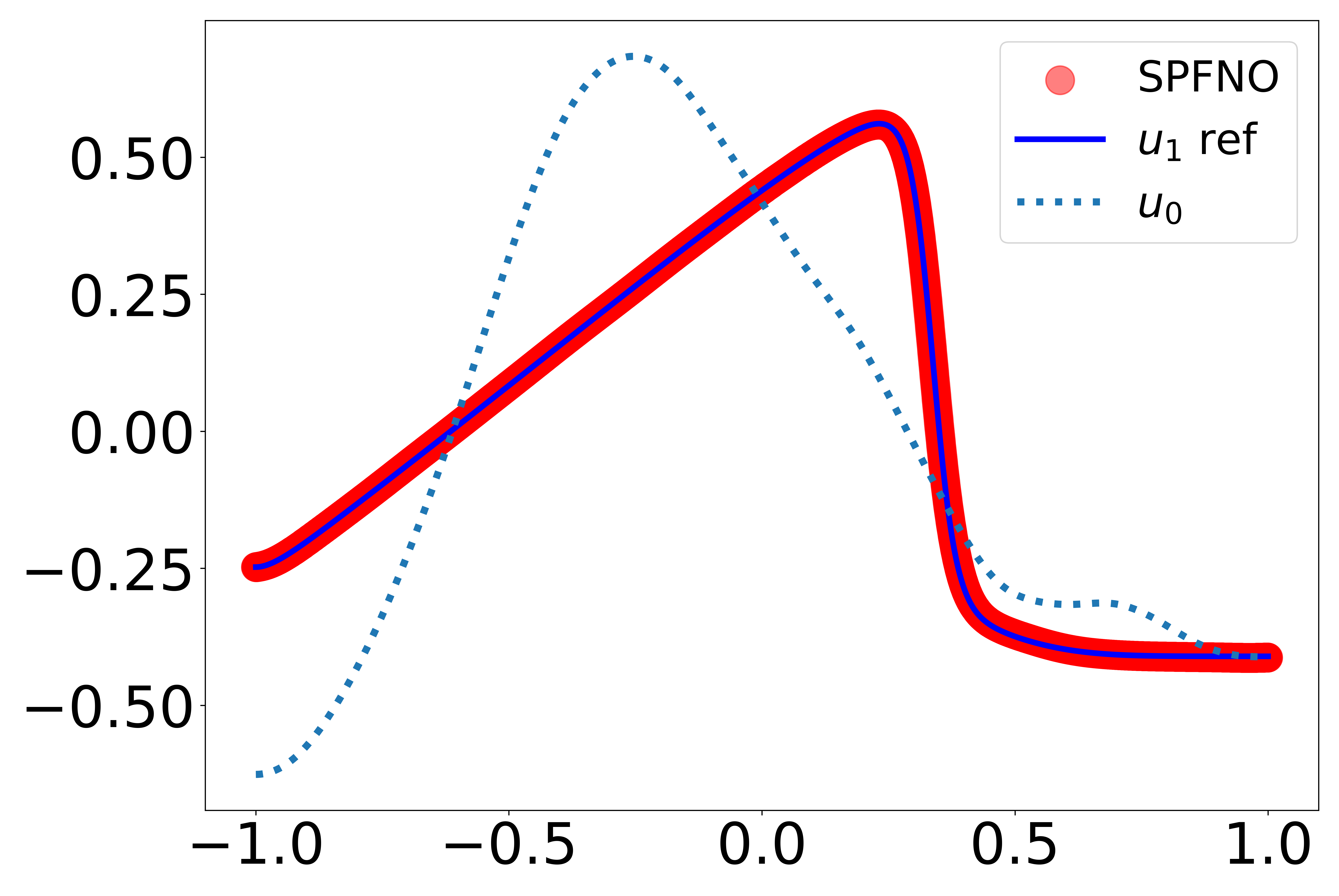}}
  \end{center}
    \caption{The prediction of SPFNO model that is trained on a $N=256$ grid but evaluated on different grids of resolution $N'$}\label{fig:eval-fine}
  \end{figure}

\section{Discussion and future work}
We presented a novel spectral operator learning (SOL) architecture for PDEs with Dirichlet and Neumann BCs. This method, named SPFNO, combines traditional spectral methods and the neural
operator architecture, so that it satisfies the corresponding BCs exactly. The BC-satisfying properties were proved both theoretically and numerically. Numerical experiments also showed that the SOL methods yield very close results regardless of the different types of grids they are associated
with. On the other hand, compared with baselines including non-BC-satisfying models and BC-satisfying BOON model, state-of-the-art performance in solving a variety of widely adopted benchmark problems
can be achieved with our proposed framework. From the perspective of machine learning, the BC-satisfying spectral structure
is an intuitive bias that effectively shrinks the hypothesis space.

Although the paper focuses on the data-driven training of neural operators, we can also directly learn the target operator by utilizing the physics constraints and minimizing the residuals of
equations, which can dramatically reduce the
dependence on datasets. Readers may find that the residuals of the BCs that are usually difficult to handle will vanish for an SOL model, making the training much easier. Further research holds promise for future exploration. And while the bandwidth of the
learnable matrix is generally considered to be determined by the orthogonality of the basis, a suitable increase of it improves the approximation capability of SPFNO, for which the theoretical analysis is an
interesting topic for further research. In addition, developing other SOL instances for more complex BCs and
geometries, e.g., the radiation BCs or unbounded domains, is an important subject of future research.

\vspace{12pt}

\section*{Declaration of generative AI and AI-assisted technologies in the writing process}

During the preparation of this work the author(s) used ERNIE Bot in order to improve language and readability. After using this tool/service, the author(s) reviewed and edited the content as needed and take(s) full responsibility for the content of the publication.

\section*{Acknowledgement}
This work was supported by the National Natural Science Foundation of China (Grant
Nos. 12271523, 11901577, 11971481, 12071481), by the Defense Science Foundation of
China (Grant No. 2021-JCJQ-JJ-0538), by the National Key R\&D Program of China (Grant
No. SQ2020YFA0709803), by the Science and Technology Innovation Program of Hunan
Province (Grant Nos. 2022RC1192, 2021RC3082), and by the Research Fund of National
University of Defense Technology (Grant Nos. ZK19-37, ZZKY-JJ-21-01).


\bibliographystyle{elsarticle-num}
\bibliography{iclr2024_conference}








\end{document}